\documentclass[11pt]{article}
\usepackage[margin=0.95in]{geometry}
\usepackage[noend]{algpseudocode}
\usepackage{algorithm}
\usepackage{times}
\usepackage{amssymb}
\usepackage{amsmath}
\usepackage{latexsym}
\usepackage{graphics}
\usepackage{url}
\usepackage{verbatim}
\usepackage{color}
\usepackage{setspace}
\usepackage{multirow}
\usepackage{ctable}
\usepackage{lineno}
\usepackage{booktabs}
\usepackage{graphicx}

\usepackage{listings}
\usepackage{amssymb}
\usepackage{latexsym}
\usepackage{graphics}
\usepackage{cite}
\usepackage{booktabs}
\usepackage{amsmath}
\usepackage{ragged2e}
\usepackage{setspace}
\usepackage[T1]{fontenc}
\usepackage[utf8]{inputenc}
\usepackage[english]{babel}
\usepackage{algorithm}
\usepackage{algpseudocode}
\usepackage{pifont}
\usepackage{graphicx}% http://ctan.org/pkg/graphicx
\usepackage{yhmath}% http://ctan.org/pkg/yhmath
\usepackage{mathdots}% http://ctan.org/pkg/mathdots
\usepackage{MnSymbol}% 
\usepackage[titletoc]{appendix}

\usepackage{mathtools}
\usepackage{color}

\algnewcommand\And{\textbf{and}}
\algnewcommand\Or{\textbf{or}}
\algnewcommand\Not{\textbf{not}}
\algnewcommand\In{\textbf{in}}
\algnewcommand\Each{\textbf{each}}

\newtheorem{theorem}{Theorem}[section]          % Numbers theorems bysection.
\newtheorem{lemma}[theorem]{Lemma}             % Numbers a lemma by a theorem number.
\newtheorem{corollary}[theorem]{Corollary}         % Numbers a corollary by a theorem number.
\newenvironment{proof}{{\em Proof.}}{\hspace*{\fill}$\Box$\par\vspace{3mm}}

\newcommand{\squishlist}{
 \begin{list}{$\bullet$}
  { \setlength{\itemsep}{0pt}
     \setlength{\parsep}{3pt}
     \setlength{\topsep}{3pt}
     \setlength{\partopsep}{0pt}
     \setlength{\leftmargin}{2.5em}
     \setlength{\labelwidth}{1em}
     \setlength{\labelsep}{0.5em} } }

\newcommand{\squishlisttwo}{
 \begin{list}{$\triangleright$}
  { \setlength{\itemsep}{0pt}
     \setlength{\parsep}{0pt}
    \setlength{\topsep}{0pt}
    \setlength{\partopsep}{0pt}
    \setlength{\leftmargin}{2em}
    \setlength{\labelwidth}{1.5em}
    \setlength{\labelsep}{0.5em} } }

\newcommand{\squishend}{
  \end{list}  }

\usepackage{listings}

\definecolor{verbgray}{gray}{0.9}

\lstnewenvironment{code}{%
  \lstset{backgroundcolor=\color{verbgray},
 frame=single,
  language=C,
  framerule=0pt,
  basicstyle=\ttfamily,
  showstringspaces=false,
  commentstyle=\color{blue}\textit,
  keywordstyle=\color{black}\bf,
  numbers=none,
  keepspaces=true,
  columns=fullflexible}}{}

\newcommand{\bS}{\mathbf{S}}

\definecolor{shadecolor}{rgb}{.91, .91, .91}
\definecolor{bordercolor}{rgb}{.8, .8, .6}

\definecolor{ultramarine}{rgb}{0, 0.125, 0.376}

 \definecolor{arsenic}{rgb}{0.23, 0.27, 0.29}
 \definecolor{beige}{rgb}{0.96, 0.96, 0.86}

\definecolor{amber}{rgb}{1.0, 0.75, 0.0}
\definecolor{orange}{rgb}{1.0, 0.49, 0.0}
\definecolor{dandelion}{rgb}{0.94, 0.88, 0.19}

  \definecolor{indiagreen}{rgb}{0.07, 0.53, 0.03}
  \definecolor{huntergreen}{rgb}{0.21, 0.37, 0.23}

\definecolor{shadecolor}{rgb}{.9, .9, .9}

 \usepackage{framed}

 \colorlet{framecolor}{ultramarine}
  \colorlet{exframecolor}{orange}

    {\endMakeFramed}

    \newenvironment{frshaded*}{%
    \MakeFramed {\advance\hsize-\width \FrameRestore}}%
    {\endMakeFramed}
    
    \newcounter{examplecounter}
\newenvironment{exam}{
 \begin{frshaded*}
    \refstepcounter{examplecounter}%
    \noindent
  \textbf{Example \arabic{examplecounter}}%
  \quad
}{%
\end{frshaded*}
}

{\endMakeFramed}

\newenvironment{frshaded2*}{%
    \MakeFramed {\advance\hsize-\width \FrameRestore}}%
    {\endMakeFramed}

\newenvironment{boxed2}{
 \begin{frshaded2*}
}{%
\end{frshaded2*}
}

\newcommand*\wc{{\mkern 2mu\cdot\mkern 2mu}}
\usepackage[affil-it]{authblk}
\newcommand{\U}{\mathcal{U}}

\newcommand{\factor}{\mathbf{sub}}

\begin{document}

\title{ \Large Constructing de Bruijn sequences by concatenating smaller universal cycles}

\author{Daniel Gabric%
  \thanks{E-mail: \texttt{dgabric@uoguelph.ca}} \hspace{0.05cm} and Joe Sawada%
  \thanks{E-mail: \texttt{jsawada@uoguelph.ca}}}
\affil{University of Guelph, Canada}

\maketitle
 
\begin{abstract}
We present sufficient conditions for when an ordering of universal cycles $\alpha_1, \alpha_2, \ldots, \alpha_m$ for disjoint sets $\bS_1, \bS_2, \ldots , \bS_m$ can be concatenated together
to obtain a universal cycle for $\bS =  \bS_1 \cup \bS_2 \cup \cdots \cup  \bS_m$.  When $\bS$ is the set of all $k$-ary strings of length $n$, the result of such a successful construction is a de Bruijn sequence.
Our conditions are applied to generalize two previously known de Bruijn sequence constructions and then they are applied to develop three new de Bruijn sequence constructions.  
\end{abstract}
%=====================================================================
%=====================================================================
\section{Introduction}  \label{sec:intro}

Let $\Sigma_k= \{0,1,\ldots , k-1\}$ be an alphabet of $k\geq 2$ symbol and let $\Sigma_k^n$ be the set of $k$-ary strings of length $n$. Given a non-empty subset $\mathbf{S}$ of $\Sigma_k^n$, a \emph{universal cycle} for $\mathbf{S}$ is a sequence of length $|\mathbf{S}|$  that contains every string in $\mathbf{S}$ as a substring exactly once when the sequence is viewed circularly. A universal cycle is said to be a \emph{de Bruijn sequence} in the case that $\mathbf{S}=\Sigma_k^n$. For example, \[ 000111222121101201002102202\] is a de Bruijn sequence for $\Sigma_3^3$. 
%It is well known that there are $k!^{k^{n-1}}/k^n$ non-isomorphic de Bruijn sequneces for $\Sigma_k^n$ \cite{vanAardenne-Ehrenfest1987}.    
It is well known that de Bruijn sequences are in one-to-one correspondence with directed Euler cycles in a related de Bruijn graph.  However, algorithms to find Euler cycles in graphs require that the graph be stored in memory, and the de Bruijn graph is exponential in size.    Amazingly, a prefer-smallest greedy approach~\cite{martin,ford} generates the lexicographically smallest de Bruijn sequence~\cite{fred-lexleast};  however, like other preference-based methods~\cite{alhakim-span}, including prefer-same~\cite{eldert,fred-nfsr} and prefer-opposite~\cite{pref-opposite} in the binary case,  it also requires an exponential amount of memory.    As a result, there has been significant research to efficiently construct de Bruijn sequences for arbitrary $n$ and $k$.  The majority of this work constructs de Bruijn sequences via a successor-rule, finding one symbol at  a time using the previous $n$ symbols.   Of this work, most apply only to the case when $k=2$ \cite{jansen, Etzion1984, Etzion1987,fred-lexleast,grandma,wong,huang}, although several approaches generalize to larger alphabets~\cite{fred-lexleast,grandma2, ETZION1986331,kwong}.  In the best case, these algorithms require $O(n)$-time per symbol and use $O(n)$-space.    

The most efficient constructions of de Bruijn sequences arise from a concatenation approach, with some generating each symbol in  $O(1)$-amortized time using $O(n)$ space.  However, very little is known about these constructions in general.  The first such construction was given by Fredericksen and Maiorana~\cite{fkm2}.  To describe their approach we need the following two definitions.  A \emph{necklace} is the lexicographically smallest string in an equivalence class under rotation.  The \emph{periodic reduction} of a string $\alpha = a_1a_2\cdots a_n$ is $a_1a_2\cdots a_p$ where $p$ is the smallest integer such that $\alpha =  (a_1a_2\cdots a_p)^{n/p}$, where exponentiation denotes concatenation.  They show that the lexicographically smallest de Bruijn sequence\footnote{They actually show (equivalently) their approach produces the lexicographically largest de Bruijn sequence.} can be constructed by concatenating together the periodic reductions of all $k$-ary necklaces of length $n$ listed in lexicographic order.  An analysis in~\cite{RSW} shows that this construction generates each symbol in $O(1)$-amortized time using $O(n)$ space.  

\begin{exam}  \small
The set of necklaces of length $n=4$ when $k=2$ is $\{0000, 0001, 0011, 0101, 0111, 1111\}$.
By concatenating the periodic reduction of each necklace in lexicographic order we obtain the following de Bruijn sequence for $\Sigma_2^4$
\[ 0 \cdot 0001 \cdot  0011 \cdot 01 \cdot 0111 \cdot 1 \]
where $\cdot$ is used to denote concatenation for clarity.
\end{exam}
\noindent
There is a very subtle point in the description of this algorithm.  Are the periodic reductions listed in lexicographic order, or are the necklaces first listed in lexicographic order and \emph{then} the periodic reductions are applied? It turns out, that when using lexicographic order, it does not matter; the two listings are equivalent.  This is pointed out by Ruskey~\cite{ruskey} who also describes the algorithm as \emph{the concatenation of all Lyndon words whose length divide $n$ in lexicographic order}, where  \emph{Lyndon words} are necklaces equal to their periodic reductions.   Interestingly, these approaches are no longer equivalent when we consider a co-lexicographic (colex) order as pointed out in~\cite{grandma}.  If we order the periodic reductions in colex order for $n=4$ and $k=2$ we obtain 
\[ 0  \cdot 1 \cdot 01 \cdot 0001 \cdot 0011 \cdot 0111 = 0101000100110111 \]
which is not a de Bruijn sequence (it has no substring 1111).  However, by first ordering the necklaces in colex order and then taking their periodic reductions, it is proved by Dragan et al.~\cite{grandma2} that for any $k\geq 2$ and $n \geq 1$ the result is a de Bruijn sequence for $\Sigma_k^n$.  When $k=2$ and $n=4$ this construction produces the following de Bruijn sequence
\[ 0  \cdot 0001 \cdot 01  \cdot 0011 \cdot 0111 \cdot 1.\]
When $k=2$, by applying an algorithm in~\cite{SAWADA201725} each symbol can be generated in $O(1)$-amortized time using $O(n)$ space.  When $k > 2$, each symbol can be produced in $O(n)$ time~\cite{grandma2}, and it remains an open problem to improve this bound.  For each of these two concatenation approaches, observe that the periodic reductions correspond to universal cycles for each necklace equivalence class.

The main result of this paper is to provide more general conditions for when universal cycles $\alpha_1, \alpha_2, \ldots , \alpha_m$ for disjoint subsets $\bS_1, \bS_2, \ldots ,\bS_m$ of $\Sigma_k^n$ can be concatenated together to obtain a universal cycle for $\bS = \bS_1 \cup \bS_2\cup  \cdots \cup  \bS_m$.  Our results can be applied to:
\begin{itemize}
\item generalize the lexicographic concatenation scheme by Fredricksen and Maiorana~\cite{fkm2}, 
\item generalize the colex concatenation scheme by Dragan et al.~\cite{grandma2},
\item obtain a new de Bruijn sequence construction that in the binary case is observed to be equivalent to the successor-rule based algorithm in~\cite{kwong}, and
\item obtain two new binary de Bruijn sequence constructions based on co-necklaces (defined in Section~\ref{subsection:two}).
\end{itemize}
These results generalize preliminary work presented at WORDS 2017~\cite{Gabric2017}.

In addition to the two concatenation schemes presented earlier, two others are known to construct universal cycles for subsets of binary strings.  One  generalizes Fredricksen and Maiorana's approach for binary strings with a minimum specified weight (number of 1s)~\cite{SAWADA201431}, and another is based on cool-lex order that applies to binary strings in a given weight range~\cite{coollex,dbrange,walcom}.  Each algorithm constructs universal cycles in $O(1)$-amortize time per bit using $O(n)$-space.  The sufficient conditions presented in this paper do not apply to these algorithms.

The remainder of this paper is presented as follows.
In Section~\ref{sec:back}, we present background definitions and notation. In Section~\ref{sec:main}, we present our main results which provide sufficient
conditions for when smaller universal cycles can be concatenated together to create a new larger universal cycle. In Section~\ref{sec:deb}, we apply our conditions to generalize previously known de Bruijn sequence constructions and to develop three new and generalized de Bruijn sequence constructions.

%=====================================================================
%=====================================================================
\section{Background Definitions and Notation} \label{sec:back}

Let $\alpha=a_1a_2\cdots a_m$ and $\beta= b_1b_2\cdots b_n$ be two distinct $k$-ary strings. Then $\alpha$ comes before $\beta$ in \emph{lexicographic} (lex) order if $\alpha$ is a proper prefix of $\beta$  or if $a_i < b_i$ for the smallest $i$ where $a_i\neq b_i$. We say that  $\alpha$ comes before $\beta$ in \emph{colexicographic} (colex) order if $\alpha$ is a proper suffix of $\beta$  or if $a_i < b_i$ for the largest $i$ where $a_i\neq b_i$.
Given a set $\bS$ of strings of arbitrary length, let
\begin{itemize}
\item  $\text{lex}(\bS)$ denote the strings of $\bS$ listed in lex order,
 \item  $\text{revlex}(\bS)$ denote the strings of $\bS$ listed in reverse lex order,
 \item  $\text{colex}(\bS)$ denote the strings of $\bS$ listed in colex order, and
 \item  $\text{revcolex}(\bS)$ denote the strings of $\bS$ listed in reverse colex order.  
\end{itemize}

\begin{exam}  \small
\label{example:three}
%\vspace{-0.05in}
%{\noindent \bf Example}\\
\noindent Let $\bS = \{ 0101, 21201, 12020,000,220, 02102 \}$.   Then 
\begin{equation}
\begin{split}
\text{lex}(\bS)&=000,0101,02102,12020,21201,220,\nonumber \\
\text{revlex}(\bS)&=220,21201,12020,02102,0101,000,\nonumber\\
\end{split}
\quad
\begin{split}
\text{colex}(\bS)&=000,12020,220,0101,21201,02102,\nonumber \\
\text{revcolex}(\bS)&=02102,21201,0101,220,12020,000.\nonumber
\end{split}
\end{equation}
\vspace{-0.1in}
\end{exam}

\noindent

Let $\alpha=a_1a_2\cdots a_s$ and $\beta=b_1b_2\cdots b_t$ be two strings with $s,t\geq n>0$.  
Let $\text{suff}_n(\alpha)$ be the length $n$ suffix of $\alpha$ and $\text{pre}_n(\alpha)$ be the length $n$ prefix of $\alpha$. For example, $\text{suff}_3(0032233)=233$ and $\text{pre}_3(0032233)=003$. The set of necklaces when $n=5$ and $k=2$ is $$\mathbf{S}=\{00000,00001,00011,00101,00111,01011,01111,11111\}.$$ When you order $\mathbf{S}$ in colex order, an interesting property between adjacent necklaces becomes apparent, \[\text{colex}(\mathbf{S}) =00000,00001,00101,00011,01011,00111,01111,11111.\] Observe that for any two adjacent necklaces $\sigma$ and $\tau$ in the above listing, if $j$ is the smallest index where $\tau$ is not $0$ at index $j$, then $\sigma$ and $\tau$ share a length $n-j$ suffix. For example, consider the adjacent necklaces $00111$ and $01111$. At index $2$, $01111$ is not $0$, and in fact this is the smallest index for which this is true. With this index, we predict that the longest matching suffix is of length $5-2=3$, which we can easily verify to be true. The following two properties generalize this idea, and extend it to prefixes.

{\bf Suffix-related:} Let $x\in \Sigma_k$, and let $j$ be the smallest index of $\beta$ such that $x \neq b_j$, or $\infty$ if no such $j$ exists. Then the ordered pair of strings $(\alpha,\beta)$ is said to be \emph{suffix-related} with respect to $(x,n)$ if $j\leq n$ and $\text{suff}_{n-j}(\alpha)=\text{suff}_{n-j}(\beta)$.

{\bf Prefix-related:} Let $x\in \Sigma_k$, and let $j$ be the smallest index of $\alpha$ such that $x\neq a_{s-j}$, or $\infty$ if no such $j$ exists. Then the ordered pair of strings $(\alpha, \beta)$ is said to be \emph{prefix-related} with respect to $(x,n)$ if $j\leq n$ and $\text{pre}_{n-j-1}(\alpha) = \text{pre}_{n-j-1}(\beta)$. 

%==============
\begin{exam}
\label{example:two}  \small
Let  $\alpha=00001200$, $\beta =02000200$, $x=0$ and $n=5$,

\smallskip

\noindent
The smallest index $j$ of $\beta$ such that $b_j \neq x$ is $j=2$.  Note, $\text{suff}_{5-2}(00001200)=\text{suff}_{5-2}(02000200)=200$. Thus, $(\alpha,\beta)$ are suffix-related with respect to $(0,5)$.

\smallskip

\noindent
The smallest $j>0$ such that $a_{s-j} \neq x$ is $j=2$.  Note, $00=\text{pre}_{5-2-1}(00001200)\neq\text{pre}_{5-2-1}(02000200)=02$. Thus, $(\alpha,\beta)$ are not prefix-related with respect to $(0,5)$.
\end{exam}
%==============

Let $\text{ext}_n(\alpha)=\alpha^t$, where $t$ is the smallest integer so  $t|\alpha|\geq n$.  
Let $\factor_n(\alpha)$ be the set of all length $n$ substrings in the cyclic string $\alpha$.   For example, $\factor_3(01201) = \{012, 120, 201, 010, 101\}$ and $\factor_5(02) = \{02020,20202\}.$

%=====================================================================
%=====================================================================
\section{Concatenating Universal Cycles}  \label{sec:main}
%In this section we present a novel concatenation construction based on co-necklaces. We also prove that it generates a de Bruijn sequence.

\begin{sloppypar}
Let $\bS$ be a non-empty subset of  $\Sigma_k^n$.  A partition of $\bS$ into subsets $\bS_1,\bS_2,\ldots, \bS_m$ is called a \emph{UC-partition} if there is a universal cycle $\alpha_i$ for each $\bS_i$, $1\leq i \leq m$.
\end{sloppypar}

\begin{exam}  \small
The following sets
\label{example:one}
  \begin{align}
       \bS_1 &= \{00000,00001,00011,00111,01111,11111,11110,11100,11000,10000\},  \nonumber \\
       \bS_2 &= \{00100,01001,10011,00110,01101,11011,10110,01100,11001,10010\},  \nonumber \\
       \bS_3 &=  \{00010,00101,01011,10111,01110,11101,11010,10100,01000,10001\},\nonumber \\
       \bS_4 &= \{01010,10101\},\nonumber        
  \end{align}
 together form a UC-partition of $\Sigma_2^5$  with universal cycles
 \[\alpha_1=0000011111, \ \alpha_2=0010011011, \ \alpha_3=0001011101, \ \alpha_4=01.\] The sets are pairwise disjoint and their union is $\Sigma_2^5$.
\end{exam}

Given a UC-partition for a set $\bS$ along with their corresponding universal cycles, we present conditions for when the smaller universal cycles can be concatenated together to obtain a universal cycle for $\bS$. 

\newpage
%--------------------------------------------------
\begin{boxed2}  
%\noindent {\bf Colex co-necklace concatenation construction} 

\vspace{-0.15in}

\begin{theorem}
Let $\bS_1, \bS_2, \ldots, \bS_m$ be a UC-partition of $\bS \subseteq \Sigma_k^n$ with universal cycles $\alpha_1,\alpha_2,\ldots, \alpha_m$ where $x$ is the first symbol in $\alpha_1$ and $\U_{m,n}=\alpha_1\alpha_2\cdots \alpha_m$. If the following three conditions hold,

\begin{enumerate}
\item $|\alpha_1|\geq n$,
\item $\alpha_1$ has a largest prefix of consective $x$'s out of all $\alpha_i,1\leq i \leq m$ ,
\item For each $1\leq i < m$, $(\text{ext}_n(\alpha_i),\text{ext}_n(\alpha_{i+1}))$ are suffix-related with respect to $(x,n)$,
\end{enumerate}
then $\U_{m,n}$ is universal cycle for $\bS$ and $\text{suff}_n(\U_{m,n})=\text{suff}_n(\text{ext}_n(\alpha_m))$.
 \label{secondTheorem}
\end{theorem}

\vspace{-0.15in}

\end{boxed2}
%
%\begin{sloppypar}
\noindent
\begin{proof}
The proof is by induction on $m$. In the base case when $m=1$, $\U_{1,n}=\alpha_1$ is a universal cycle for $\bS_1$, and by assumption $|\alpha_1|\geq n$ so $\text{suff}_n(\U_{1,n})=\text{suff}_n(\text{ext}_n(\alpha_1))$.
Inductively, assume $\U_{m-1,n}$ is a universal cycle for $\bS-\bS_m$, and $\text{suff}_n(\U_{m-1,n})=\text{suff}_n(\text{ext}_n(\alpha_{m-1}))$, for $m>1$. 
Consider $\U_{m,n}=\U_{m-1,n} \alpha_{m}$. Let:
\begin{itemize}
\item $\alpha_1=g_1g_2\cdots g_q,$
\item $\text{suff}_n(\text{ext}_n(\alpha_{m-1})) = a_1a_2\cdots a_n,$
\item $\text{ext}_n(\alpha_{m})=x^{j-1}b_{j}b_{j+1}\cdots b_{s},$
\end{itemize}
 where $j$ is the smallest index where $b_{j}\neq x$. 
First we show that $\text{suff}_n(\U_{m,n})=\text{suff}_n(\text{ext}_n(\alpha_{m}))$. If $|\alpha_{m}|\geq n$, clearly $\text{suff}_n(\alpha_{m})$ is a suffix of $\U_{m-1,n}\alpha_{m}$. If $|\alpha_{m}|<n$, $\text{suff}_n(\text{ext}_n(\alpha_{m-1}))$ appears as a suffix of $\U_{m-1,n}$ by the inductive hypothesis. By assumption $\text{suff}_{n-j}(\text{ext}_n(\alpha_{m}))=\text{suff}_{n-j}(\text{ext}_n(\alpha_{m-1}))$, so a suffix of $\U_{m,n}$ will be $\beta = \text{suff}_{n-j}(\text{ext}_n(\alpha_{m})) \alpha_m$, which shares a suffix of length $\min(|\beta|, |\text{ext}_n(\alpha_{m})|)$ with $\text{ext}_n(\alpha_{m})$. 
%We can write $\text{suff}_{n-j}(\text{ext}_n(\alpha_{m-1})) \alpha_m$ as $\beta (\alpha_m)^p\alpha_m$ where $n-j-|\alpha_m| \leq p|\alpha_m| < n-j$, and $\beta$ is some suffix of $\alpha_m$ such that $|\beta| + p|\alpha_m| = n-j$. By writing it in this way, it is clear that $\text{suff}_{n-j}(\text{ext}_n(\alpha_{m-1})) \alpha_m$ is a suffix of $\text{ext}_n(\alpha_{m})$.
Since $j\leq |\alpha_m|$, it must be the case that $|\beta|=n-j + |\alpha_m|\geq n$, and thus $\min(|\beta|, |\text{ext}_n(\alpha_{m})|)\geq n$. Therefore, $\text{suff}_n(\U_{m,n})=\text{suff}_{n}(\text{ext}_n(\alpha_{m}))$. 
Now we prove that $\U_{m,n}$ is a universal cycle for $\bS$.
By the inductive hypothesis, $\U_{m,n}$ will
contain all of the strings in $\bS-\bS_m$ except for possibly the strings $\{ a_{2}a_{3}\cdots a_{n}g_1, a_{3}a_{4}\cdots a_{n}g_1g_2, \ldots  , a_{n}g_1\cdots g_{n-1}\}$ which were involved in the wraparound.  However, we know that $\text{suff}_{n-j}(\text{ext}_n(\alpha_{m}))=\text{suff}_{n-j}(\text{ext}_n(\alpha_{m-1}))$ and  $\text{suff}_{n}(\text{ext}_n(\alpha_{m}))=\text{suff}_{n}(\U_{m,n})$. This implies that each string in $\{a_{j+1}a_{j+2}\cdots a_ng_1\cdots g_{j}, a_{j+2}a_{j+3}\cdots a_ng_1\cdots g_{j+1}, \ldots, a_n g_1\cdots g_{n-1}\}$ occurs as a substring in the wrap-around of the cyclic $\U_{m,n}$. Furthermore, the strings $\{a_{2}a_{3}\cdots a_n x,a_{3}a_{4}\cdots a_n xx, \ldots, a_{j}a_{j+1}\cdots a_n x^{j-1}\}$ exist within $\U_{m,n}$ because $\alpha_m$ has prefix $x^{j-1}=g_1\cdots g_{j-1}$. Thus, the cyclic $\U_{m,n}$ contains each string in $\bS - \bS_m$ as a substring.
%we know that $\alpha_1$ and $\beta$ share a length $j$ prefix. So all of the strings within $\U_{m,n}$ exist within $\U_{m+1,n}$. 
Finally, we show that all strings in $\bS_{m}$  occur as a substring in $\U_{m,n}$ (when considered cyclicly). Those that are not trivially substrings of $\alpha_m$ occur either in the wrap-around or have their prefix as a suffix in $\U_{m-1,n}$ and suffix in a prefix of $\alpha_m$. Let $t=s-n$, $i=t+j$, and
\begin{itemize}
\item $\mathbf{T}_1=\{b_{t+2}b_{t+3}\cdots b_s x, b_{t+3}b_{t+4}\cdots b_s xx, \ldots, b_{i}b_{i+1}\cdots b_s x^{j-1}\}$,
\item $\mathbf{T}_2=\{b_{i+1}b_{i+2}\cdots b_sx^{j-1}b_j, b_{i+2}b_{i+3}\cdots b_sx^{j-1}b_jb_{j+1}, \ldots, b_{z+1}\cdots b_nx^{j-1}b_j b_{i+1}\cdots b_{z}\}$,
\end{itemize}
where $z = n$ if $|\alpha_m| > n$ and $z = |\alpha_m|$ otherwise.
Notice that $\mathbf{T}_1$ and $\mathbf{T}_2$ together cover all length $n$ substrings in the wraparound of $\alpha_m$.  Each string in $\mathbf{T}_1$ occurs in the wraparound of  $\U_{m,n}$ since $\text{suff}_{n}(\text{ext}_n(\alpha_{m}))=\text{suff}_{n}(\U_{m,n})$, and the prefix of $\alpha_1$ has a run of $x$ at least as big as the run in the prefix of $\alpha_m$. Since $\text{suff}_{n-j}(\text{ext}_n(\alpha_m))=\text{suff}_{n-j}(\text{ext}_n(\alpha_{m-1}))$, each string in $\mathbf{T}_2$ has a prefix in $\U_{m-1,n}$ and a suffix in $\alpha_m$, so each string in $\mathbf{T}_2$ occurs as a substring of $\U_{m,n}$. We have shown that $\bS\subseteq\factor_n(\U_{m,n})$ and by construction $|\U_{m,n}|=|\bS|$. 
Thus, $\U_{m,n}$ is a universal cycle for $\bS$.
%By induction, the theorem is proven.
\end{proof}
%\end{sloppypar}

%------------------------
\begin{exam}  \small
From Example~\ref{example:one}, the UC-partition $\bS_1, \bS_2, \bS_3, \bS_4$ of $\Sigma_2^5$ with universal cycles $\alpha_1, \alpha_2, \alpha_3, \alpha_4$  satisfies the conditions of Theorem~\ref{secondTheorem}.  Thus, the concatenation 
\[ \alpha_1\alpha_2\alpha_3\alpha_4 = 0000011111  \cdot 0010011011  \cdot 0001011101 \cdot 01\]
is a universal cycle for $\Sigma_2^5$.
\end{exam}
%----------------------

 Let $\text{rev}(\alpha)$ denote the reverse of the string $\alpha$. Let $\alpha_1,\alpha_2,\ldots , \alpha_m$ be a list of strings of length at least $n$ where $(\alpha_i,\alpha_{i+1})$ are prefix-related with respect to some $(x,n)$ for $1\leq i < m$. Then $(\text{rev}(\alpha_{i+1}),\text{rev}(\alpha_i))$ are suffix-related with respect to $(x,n)$. Thus Corollary~\ref{corollaryOne} follows from Theorem~\ref{secondTheorem}.

\begin{boxed2}  
%\noindent {\bf Colex co-necklace concatenation construction} 

\vspace{-0.15in}

\begin{corollary}
Let $\bS_1, \bS_2, \ldots, \bS_m$ be a UC-partition of $\bS \subseteq \Sigma_k^n$ with universal cycles $\alpha_1,\alpha_2,\ldots, \alpha_m$ where $x$ is the last symbol in $\alpha_m$ and $\U_{m,n}=\alpha_1\alpha_2\cdots \alpha_m$. If the following three conditions hold,

\begin{enumerate}
\item $|\alpha_m|\geq n$,
\item $\alpha_m$ has a largest suffix of consecutive $x$'s out of all $\alpha_i,1\leq i \leq m$,
\item For each $1\leq i < m$, $(\text{ext}_n(\alpha_i),\text{ext}_n(\alpha_{i+1}))$ are prefix-related with respect to $(x,n)$,
\end{enumerate}
then $\U_{m,n}$ is universal cycle for $\bS$ and $\text{pre}_n(\U_{m,n})=\text{pre}_n(\text{ext}_n(\alpha_1))$.
 \label{corollaryOne}
\end{corollary}

\vspace{-0.15in}

\end{boxed2}

%=====================================================================
%=====================================================================

\section{New Universal Cycle Concatenation Constructions }\label{sec:deb}
In this section we apply the results from the previous section to produce five new universal cycles based on concatenating together smaller universal cycles.  A direct consequence
of each result is a unique de Bruijn sequence construction; the first two were previously known and the last three are new.

Each of our constructions follows the approach outlined in Section~\ref{sec:intro} of concatenating the periodic reductions of a listing of strings.
The  function $\text{UC}$ is defined on a listing of strings $\mathcal{L} = \alpha_1, \alpha_2, \ldots , \alpha_j$ as follows, where
$pr(\alpha)$ is the periodic reduction of $\alpha$:
\[ \text{UC}(\mathcal{L}) = pr(\alpha_1) pr(\alpha_2) \cdots pr(\alpha_j).\]

%===============
\begin{exam}
If $\mathcal{L} = 1111, 1212, 13213, 23131, 32312, 331331$, then 
\begin{align}
\text{UC}(\mathcal{L})&= 1\cdot 12\cdot 13213\cdot 23131\cdot 32312\cdot 331.\nonumber 
\end{align}

\vspace{-0.15in}
\end{exam}

In the next subsection we outline three concatenation constructions based on necklaces.  Then,  we define co-necklaces and use them to outline two more constructions.

\subsection{Necklaces}\label{subsection:one}
\begin{sloppypar}
Recall that a necklace is the lexicographically smallest string in an equivalence class of strings under rotation. Let $\mathbf{Neck}_k(n)$ denote the set of $k$-ary necklaces of length $n$. For example, $\mathbf{Neck}_3(3)=\{000,001,002,011,012,021,022,111,112,122,222\}$. It is well known that $\{\factor_n(\alpha) : \alpha \in \mathbf{Neck}_k(n)\}$ is a partition of $\Sigma_k^n$. 
By applying the results from Section~\ref{sec:main}, we obtain three de Bruijn sequence constructions that are generalized to some subsets of $\Sigma_k^n$.  

%We now provide three simple proofs for concatenation constructions based on necklaces. %In Table~\ref{table:example1} we give examples of our concatenation constructions for necklaces. 

\subsubsection{Necklaces in Lex Order}

%%%THEOREM2

\begin{boxed2}
\vspace{-0.15in}
\begin{theorem}\label{thm:lex}
For $n \geq 2$ and $m\geq 2$, let $\alpha_1,\alpha_2,\ldots, \alpha_m$ be the last $m$ strings in $\text{lex}(\mathbf{Neck}_k(n))$.  Then $\mathcal{U} = \text{UC}(\alpha_1,\alpha_2,\ldots, \alpha_m)$ is a universal cycle for $\bS  = \bigcup_{i=1}^m \factor_n(\alpha_i)$.
\end{theorem}

\vspace{-0.15in}
\end{boxed2}
\noindent
\begin{proof}
Observe that $\alpha_m=(k-1)^n$ and $\alpha_{m-1}=(k-2)(k-1)^{n-1}$, so the last symbol in the sequence $\mathcal{U}$ is $x=k-1$. Note that $\alpha_{m-1}'=pr(\alpha_{m-1})pr(\alpha_m)=(k-2)(k-1)^n$ is a universal cycle for $\bS_{m-1} =\factor_n(\alpha_{m-1})\cup \factor_n(\alpha_m)$. Let $\bS_i = \factor_n(\alpha_{i})$ and $\alpha_i'=pr(\alpha_i)$ for $1\leq i<m-1$. Then $\bS_1,\bS_2,\ldots, \bS_{m-1}$ is a UC-partition of $\bS$ with universal cycles $\alpha_1',\alpha_2',\ldots, \alpha_{m-1}'$ where $x=k-1$. To prove that $\mathcal{U}$
is a universal cycle for $\bS$  we show that the three conditions of Corollary~\ref{corollaryOne} hold. 
\begin{enumerate}
\item Clearly $|\alpha_{m-1}'|\geq n$.
\item $\alpha_{m-1}'$ has suffix $x^n$, which must be maximal since all of the universal cycles are disjoint.%Here we have $x=k-1$. We must show that $x^{n+1}$ doesn't exist as a substring of $\U$. Assume that $x^{n+1}$ exists as a substring of $\U$, then it must occur between two necklaces in the ordering. It is enough to show that no necklace other than $x^n$ contains $x$ as it's first symbol. Assume there exists a necklace of the form $x\beta\neq x^n$, then $\beta x$ is smaller, so $x\beta$ is not a necklace, a contradiction. 

\item We must show that consecutive strings in $\text{ext}_n(\alpha_1'),\text{ext}_n(\alpha_2'),\ldots, \text{ext}_n(\alpha_{m-1}')$ are prefix-related with respect to $(x,n)$. Notice that $\alpha_{m-1}$ is a length $n$ prefix of $\text{ext}_n(\alpha_{m-1}')$ and $\text{ext}_n(\alpha_{i}') = \alpha_{i}$ for $1\leq i<m-1$. So we only need to show that $(\alpha_i,\alpha_{i+1})$ are prefix-related with respect to $(x,n)$ for $1 \leq i < m-1$. The proof is by contradiction. Let $\alpha_i = a_1a_2\cdots a_n$ and $\alpha_{i+1} = b_1b_2\cdots b_n$ for some $1\leq i <m-1$. Let $j$ be the smallest index of $\alpha_i$ such that $a_{n-j}\neq x$. Suppose  $a_1a_2 \cdots a_{n-j-1} \neq b_1b_2 \cdots b_{n-j-1}$. Then there exists some smallest $s<n-j$ such that $a_{s}\neq b_{s}$. Since $\alpha_i$ comes before $\alpha_{i+1}$ in lex order, then $a_{s}<b_{s}$. However, since $\alpha_i$ is a necklace, then $\gamma = a_1a_2\cdots a_{s}(k-1)^{n-s}$ will also be a necklace. But this means that $\gamma$ comes between $\alpha_i$ and $\alpha_{i+1}$ in lex order, which is a contradiction. So  $a_1a_2 \cdots a_{n-j-1} = b_1b_2 \cdots b_{n-j-1}$, which implies $\text{pre}_{n - j -1}(\alpha_i) = \text{pre}_{n - j -1}(\alpha_{i+1})$. Thus the pair of strings $(\alpha_i,\alpha_{i+1})$ are prefix-related with respect to $(x,n)$.
\end{enumerate}

\vspace{-0.2in}
\end{proof}

When $m=|\mathbf{Neck}_k(n)|$, the above theorem yields the following corollary which describes an equivalent construction to the (lexicographically smallest) de Bruijn sequence construction from~\cite{fkm2}.

\begin{corollary}\label{corollary:lex}
For $n\geq 2$, $\text{UC}(\text{lex}(\mathbf{Neck}_k(n)))$ is a de Bruijn sequence for $\Sigma_k^n$.
\end{corollary}
%\newpage
%@@@@@@@@@@@@@@
\begin{exam} \small Consider the set $\mathbf{Neck}_2(6)$ listed in lex order:

\vspace{-0.25in}

\begin{align}
&000000,000001,000011,000101,000111,001001,001011,001101,001111,010101,010111,011011,011111,111111.\nonumber
\end{align}

\vspace{-0.10in}

\noindent
By Corollary~\ref{corollary:lex}, the following is a de Bruijn sequence,

\vspace{-0.25in}

\[0\cdot 000001\cdot 000011\cdot 000101\cdot 000111,001\cdot 001011\cdot 001101\cdot 001111\cdot 01\cdot 010111\cdot 
011\cdot 011111\cdot 1.\]

\vspace{-0.10in}

\noindent
By Theorem~\ref{thm:lex}, $ 011\cdot 011111\cdot 1$,
is a universal cycle for $\factor_6(011) \cup \factor_6(011111)\cup \factor_6(1)$.
\end{exam}
%@@@@@@@@@@@@@@

\begin{comment}
%@@@@@@@@@@@@@@
\begin{exam} \small Consider the set $\mathbf{Neck}_2(6)$ listed in lex order:
\begin{align}
&000000,000001,000011,000101,000111,001001,001011,\nonumber \\
&001101,001111,010101,010111,011011,011111,111111.\nonumber
\end{align}
By Corollary~\ref{corollary:lex}, the following is a de Bruijn sequence:
\[0\cdot 000001\cdot 000011\cdot 000101\cdot 000111,001\cdot 001011\cdot 001101\cdot 001111\cdot 01\cdot 010111\cdot 
011\cdot 011111\cdot 1.\]
%
By Theorem~\ref{thm:lex}, $ 011\cdot 011111\cdot 1$,
is a universal cycle for $\factor_6(011) \cup \factor_6(011111)\cup \factor_6(1)$.
\end{exam}
%@@@@@@@@@@@@@@
\end{comment}
\subsubsection{Necklaces in Colex Order}

%%%THEOREM1
\begin{boxed2}
\vspace{-0.15in}
\begin{theorem}\label{thm:colex}
For $n \geq 2$ and $m\geq 2$, let $\alpha_1,\alpha_2,\ldots, \alpha_m$ be the first $m$ strings in $\text{colex}(\mathbf{Neck}_k(n))$.  Then $\mathcal{U} = \text{UC}(\alpha_1,\alpha_2,\ldots, \alpha_m)$ is a universal cycle for $\bS  = \bigcup_{i=1}^m \factor_n(\alpha_i)$.
\end{theorem}

\vspace{-0.15in}
\end{boxed2}
\noindent
\begin{proof}
Observe that $\alpha_1=0^n$ and $\alpha_2=0^{n-1}1$, so the first symbol in the sequence $\mathcal{U}$ is $x=0$. Note that $\alpha_1'=pr(\alpha_{1})pr(\alpha_2)=01^n$ is a universal cycle for $\bS_1 =\factor_n(\alpha_{1})\cup \factor_n(\alpha_2)$. Let $\bS_i = \factor_n(\alpha_{i+1})$ and $\alpha_i'=pr(\alpha_{i+1})$ for $2\leq i< m$. Then $\bS_1,\bS_2,\ldots, \bS_{m-1}$ is a UC-partition of $\bS$ with universal cycles $\alpha_1',\alpha_2',\ldots, \alpha_{m-1}'$ where $x=0$. To prove that $\mathcal{U}$
is a universal cycle for $\bS$ we show that the three conditions of Theorem~\ref{secondTheorem} hold. 
\begin{enumerate}
\item Clearly $|\alpha_1'|\geq n$.
\item $\alpha_1'$ has prefix $x^n$, which must be maximal since all of the universal cycles are disjoint.
\item We must show that consecutive strings in $\text{ext}_n(\alpha_1'),\text{ext}_n(\alpha_2'),\ldots, \text{ext}_n(\alpha_{m-1}')$ are suffix-related with respect to $(x,n)$. Notice that $\alpha_2$ is a length $n$ suffix of $\text{ext}_n(\alpha_1')$ and $\text{ext}_n(\alpha_{i}') = \alpha_{i+1}$ for $2\leq i<m$. So we only need to show that $(\alpha_i,\alpha_{i+1})$ are suffix-related with respect to $(x,n)$ for $2\leq i < m$. The proof is by contradiction. Let $\alpha_i = a_1a_2\cdots a_n$ and $\alpha_{i+1} =b_1b_2\cdots b_n$ for some $1< i < m$. Let $j$ be the smallest index of $\alpha_{i+1}$ such that $b_j\neq x$. Suppose  $a_{j+1}a_{j+2} \cdots a_n \neq b_{j+1}b_{j+2}\cdots b_n$. Then there exists some largest $s> j$ such that $a_s \ne b_s$.  Since $\alpha_i$ comes before $\alpha_{i+1}$ in colex order, then $a_s <b_s$. However, since $\alpha_{i+1}$ is a necklace, then $\gamma = 0^{s-1} b_sb_{s+1}\cdots b_n$ will also be a necklace. But this means that $\gamma$ comes between $\alpha_i$ and $\alpha_{i+1}$ in colex order, which is a contradiction.  So $a_{j+1}a_{j+2} \cdots a_n = b_{j+1}b_{j+2}\cdots b_n$, which implies $\text{suff}_{n-j}(\alpha_i)=\text{suff}_{n-j}(\alpha_{i+1})$. Thus the pair of strings $(\alpha_i,\alpha_{i+1})$ are suffix-related with respect to $(x,n)$.
\end{enumerate}

\vspace{-0.2in}
\end{proof}

When $m=|\mathbf{Neck}_k(n)|$, the above theorem yields the following corollary which describes a construction equivalent to the de Bruijn sequence construction from~\cite{grandma2}.

\begin{corollary}\label{corollary:colex}
For $n\geq 2$, $\text{UC}(\text{colex}(\mathbf{Neck}_k(n)))$ is a de Bruijn sequence for $\Sigma_k^n$.
\end{corollary}

\begin{exam} \small Consider the set $\mathbf{Neck}_2(6)$ listed in colex order:

\vspace{-0.25in}

\begin{align}
&000000,000001,001001,000101,010101,001101,000011,001011,011011,000111,010111,001111,011111,111111.\nonumber
\end{align}

\vspace{-0.10in}

\noindent
By Corollary~\ref{corollary:colex}, the following is a de Bruijn sequence:

\vspace{-0.25in}

\[0\cdot 000001\cdot 001\cdot 000101\cdot 01 \cdot 001101\cdot 000011\cdot 001011\cdot 011\cdot 000111\cdot 010111\cdot 001111\cdot 011111\cdot 1.\]

\vspace{-0.10in}

\noindent
By Theorem~\ref{thm:colex}, $ 0\cdot 000001\cdot 001$,
is a universal cycle for $\factor_6(0) \cup \factor_6(000001)\cup \factor_6(001)$.
\end{exam}

\subsubsection{Rotations of Necklaces in Reverse Lex Order}

Consider $\mathbf{Neck}_2(5)$ listed in reverse lex order:  $L = 11111, 01111, 01011, 00111, 00101, 00011, 00001, 00000$.   Observe that $\text{UC}(L) = 10111101011001110010100011000010$ is not a de Bruijn sequence since, when considered in a cyclic way, it does not contain the substring 00000.  However by looking at specific rotations of these strings in reverse lex order we obtain positive results.
Let $\mathbf{R}_k(n)$ be the set of all $\alpha = a_1a_2\cdots a_n$ such that  $a_{i+1}a_{i+2}\cdots a_n a_1\cdots a_{i}$ is in $\mathbf{Neck}_k(n)$, where $i$ is the largest index of $\alpha$ such that $a_i\neq 0$. For example, $\mathbf{R}_2(5) =
\{11111,  11110, 10110, 11100, 10100, 11000, 10000, 00000\}$.

%\{00000, 10000, 11000, 10100, 11100, 10110, 11110, 11111\}$.

%%%THEOREM3
\begin{boxed2}
\vspace{-0.15in}
\begin{theorem}\label{thm:neckRev}
For $n \geq 2$ and $m\geq 2$, let $\alpha_1,\alpha_2,\ldots, \alpha_m$ be the last $m$ strings in $\text{revlex}(\mathbf{R}_k(n))$.  Then $\mathcal{U} = \text{UC}(\alpha_1,\alpha_2,\ldots, \alpha_m)$ is a universal cycle for $\bS  = \bigcup_{i=1}^m \factor_n(\alpha_i)$.
\end{theorem}

\vspace{-0.15in}
\end{boxed2}
\noindent
\begin{proof}
Observe that $\alpha_m=0^n$ and $\alpha_{m-1}=10^{n-1}$, so the last symbol in the sequence $\mathcal{U}$ is $x=0$. Note that $\alpha_{m-1}'=pr(\alpha_{m-1})pr(\alpha_m)=10^n$ is a universal cycle for $\bS_{m-1} =\factor_n(\alpha_{m-1})\cup \factor_n(\alpha_m)$. Let $\bS_i = \factor_n(\alpha_{i})$ and $\alpha_i' = pr(\alpha_i)$ for $1\leq i< m-2$. Then $\bS_1,\bS_2,\ldots, \bS_{m-1}$ is a UC-partition of $\bS$ with universal cycles $\alpha_1',\alpha_2',\ldots, \alpha_{m-1}'$ where $x=0$. To prove that $\mathcal{U}$
is a universal cycle for $\bS$  we show that the three conditions of Corollary~\ref{corollaryOne} hold.
\begin{enumerate}
\item Clearly $|\alpha_{m-1}'|\geq n$.
\item $\alpha_{m-1}'$ has suffix $x^n$, which must be maximal since all of the universal cycles are disjoint.
\item We must show that consecutive strings in $\text{ext}_n(\alpha_1'),\text{ext}_n(\alpha_2'),\ldots, \text{ext}_n(\alpha_{m-1}')$ are prefix-related with respect to $(x,n)$. Notice that $\alpha_{m-1}$ is a length $n$ prefix of $\text{ext}_n(\alpha_{m-1}')$ and $\text{ext}_n(\alpha_{i}') = \alpha_{i}$ for $1\leq i<m-1$. So we only need to show that $(\alpha_i,\alpha_{i+1})$ are prefix-related with respect to $(x,n)$ for $1\leq i < m-1$. The proof is by contradiction. Let $\alpha_i=a_1a_2\cdots a_n$ and $\alpha_{i+1} = b_1b_2\cdots b_n$ for some $1\leq i < m-1$. Let $j$ be the smallest index of $\alpha_i$ such that $a_{n-j}\neq x$. Suppose  $a_1a_2 \cdots a_{n-j-1} \neq b_1b_2 \cdots b_{n-j-1}$. Then there exists some smallest $s<n-j$ such that $a_{s}\neq b_{s}$. Since $\alpha_i$ comes before $\alpha_{i+1}$ in reverse lex order, then $a_{s}>b_{s}$. Let $\gamma=a_1a_2\cdots a_{s}0^{n-s}$. Clearly $\gamma$ is between $\alpha_i$ and $\alpha_{i+1}$ in reverse lex order, and $0^{n-s}a_1a_2\cdots a_{s}$ is a necklace since $0^{n-s}$ is the largest run of zeroes within the string, a contradiction. So $a_1a_2\cdots a_{n-j-1} = b_1b_2\cdots b_{n-j-1}$, which implies $\text{pre}_{n - j -1}(\alpha_i) = \text{pre}_{n - j -1}(\alpha_{i+1})$. Thus the pair of strings $(\alpha_i,\alpha_{i+1})$ are prefix-related with respect to $(x,n)$.
\end{enumerate}

\vspace{-0.2in}

\end{proof}
\begin{corollary}\label{corollary:neckRev}
For $n\geq 2$, $\text{UC}(\text{revlex}(\mathbf{R}_k(n)))$ is a de Bruijn sequence for $\Sigma_k^n$.
\end{corollary}
\end{sloppypar}

The concatenation scheme described in the above corollary was originally motivated by considering the successor-rule based construction in~\cite{kwong}.    Although we do not prove it here,
the two constructions produce the same de Bruijn sequences when $k=2$, but produce different sequences for $k>2$.

\begin{exam} \small Consider the set $\mathbf{R}_2(6)$ listed in reverse lex order:

\vspace{-0.25in}

\begin{align}
&111111,111110,111100,111000,110110,110100,110000,101110,101100,101010,101000,100100,100000,000000.\nonumber
\end{align}

\vspace{-0.10in}

\noindent
By Corollary~\ref{corollary:neckRev}, the following is a de Bruijn sequence:

\vspace{-0.25in}

\[1 \cdot 111110\cdot 111100\cdot 111000\cdot 110\cdot 110100\cdot 110000\cdot 101110\cdot 101100\cdot 10\cdot 101000\cdot 100\cdot 100000\cdot 0.\]

\vspace{-0.10in}

\noindent
By Theorem~\ref{thm:neckRev}, $ 100\cdot 100000\cdot 0$,
is a universal cycle for $\factor_6(100) \cup \factor_6(100000)\cup \factor_6(0)$.
\end{exam}

\begin{comment}
\begin{table} [h]
  \begin{center} \small
   \caption{ Three de Bruijn sequences for $\Sigma_4^3$ using the necklace concatenation constructions (listed in the order they were proven).}
   
\label{table:example1}  
 \begin{tabular} {c    }
 {{\bf de Bruijn sequences for $\Sigma_4^3$}} \\  \hline

  $1\wc 112\wc 113 \wc 114 \wc 122 \wc 123 \wc 124 \wc 132 \wc 133 \wc 134 \wc 142 \wc 143 \wc 144 \wc 2 \wc 223 \wc 224 \wc 233 \wc 234 \wc 243 \wc 244 \wc 3 \wc 334 \wc 344 \wc 4$ \\
 $1\wc 112\wc 122\wc 2\wc 132\wc 142\wc 113\wc 123\wc 223\wc 133\wc 233\wc 3\wc 143\wc 243\wc 114\wc 124\wc 224\wc 134\wc 234\wc 334\wc 144\wc 244\wc 344\wc 4$  \\
$4\wc 441\wc 431\wc 421\wc 411\wc 344\wc  341\wc 334\wc 3\wc 331\wc 321\wc 311\wc 244\wc 243\wc  241\wc 234\wc 233\wc  231\wc 224\wc 223\wc 2\wc 221\wc 211\wc 1$
			 %$3\cdot 331\cdot 321\cdot 311\cdot 233\cdot 231\cdot 223\cdot 2\cdot 221\cdot 211\cdot1$

%1001000101000011101001101010110110010111001111110111100011000000
\end{tabular}
\end{center}
\end{table}
\end{comment}

%====================================
%====================================

\subsection{Co-necklaces}\label{subsection:two}

For this subsection we will be working over the binary alphabet $\Sigma_2=\{0,1\}$.  Let $\alpha$ be a binary string and let $\overline{\alpha}$ denote its bitwise complement.  We say that $\alpha$ is a \emph{co-necklace} if $\alpha\overline{\alpha}$ is a necklace.  The set of all co-necklaces of length 5 is $\{00000, 00010, 00100, 010101\}$.   If $\alpha$ is a co-necklace, then we call $\alpha\overline{\alpha}$ an \emph{extended co-necklace}. Let $\mathbf{coN}(n)$ denote the set of all extended co-necklaces for  co-necklaces of length $n$. For example, $\mathbf{coN}(5)=\{0000011111,0001011101, 0010011011,0101010101\}$. It is well known that $\{ \factor_n(\alpha):\alpha \in \mathbf{coN}(n)\}$ is a partition of $\Sigma_2^n$; they correspond to the partition obtained from the complemented cycling register~\cite{golomb}.
%We now provide two simple proofs for concatenation constructions based on co-necklaces. In Table~\ref{table:example2} we give examples of our concatenation constructions for co-necklaces.
Using extended co-necklaces, we apply Theorem~\ref{secondTheorem} and Corollary~\ref{corollaryOne} to construct new universal cycles for subsets of $\Sigma_2^n$ and ultimately to produce two new binary de Bruijn sequence constructions.   Unlike necklaces, we cannot simply use lex or colex orderings.  For example,  neither 
\[ \text{UC}(\text{lex}(\mathbf{coN}(5))) = 0000011111 \cdot 0001011101 \cdot 0010011011 \cdot 01,  \text{ or} \]
\[ \text{UC}(\text{colex}(\mathbf{coN}(5))) = 01 \cdot 0001011101   \cdot 0010011011  \cdot    0000011111 \]
 are de Bruijn sequences, since neither contains the substring 10101.

%===================================================
\subsubsection{Extended Co-necklaces in Reverse Colex Order}

The following lemma will be useful in the proof of our universal cycle construction using the reverse colex order of extended co-necklaces.

\begin{lemma}\label{lemma:one}
Let $\alpha =a_1a_2\cdots a_{2n}$ and $\alpha'=b_1b_2\cdots b_{2n}$ be consecutive strings in the listing $\text{revcolex}(\mathbf{coN}(n))$. Then $\text{pre}_n(\alpha)$ comes before $\text{pre}_n(\alpha')$ in colex order.
\end{lemma}
\begin{proof}
Since $\alpha,\alpha'\in \mathbf{coN}(n)$,  $\alpha=\beta\overline{\beta}$ and $\alpha'=\beta'\overline{\beta'}$ for some $\beta,\beta'\in \Sigma_2^n$. Because $\alpha$ comes before $\alpha'$ in reverse colex order and $\overline{\beta}\neq \overline{\beta'}$, there exists a largest index $n+1 \leq i\leq 2n$ where $a_i\neq b_i$ and $a_i > b_i$. This implies that the largest index $1\leq j\leq n$ such that $a_j\neq b_j$ is $j=i-n$. This means that $a_j = \overline{a_i}$ and $b_j=\overline{b_i}$ and thus $a_j < b_j$. Therefore $a_1a_2\cdots a_n = \beta = \text{pre}_n(\alpha)$ comes before $b_1b_2\cdots b_n = \beta' =\text{pre}_n(\alpha')$ in colex order.
\end{proof}

\noindent
%%%THEOREM4
\begin{boxed2}
\vspace{-0.15in}
\begin{theorem}\label{thm:coNeckColex}
For $n \geq 2$ and $m\geq 1$, let $\alpha_1,\alpha_2,\ldots, \alpha_m$ be the first $m$ strings in $\text{revcolex}(\mathbf{coN}(n))$.  Then $\mathcal{U} = \text{UC}(\alpha_1,\alpha_2,\ldots, \alpha_m)$ is a universal cycle for $\bS  = \bigcup_{i=1}^m \factor_n(\alpha_i)$.
\end{theorem}

\vspace{-0.15in}
\end{boxed2}
\noindent
\begin{proof}
Observe that $\alpha_1=0^n1^n$, so the first symbol in the sequence $\mathcal{U}$ is $x=0$.
Let $\bS_i = \factor_n(\alpha_{i})$ and $\alpha_i'=pr(\alpha_i)$ for $1 \leq i\leq m$. Then $\bS_1,\bS_2,\ldots, \bS_{m}$ is a UC-partition of $\bS$ with universal cycles $\alpha_1',\alpha_2',\ldots, \alpha_m'$ where $x=0$. To prove that $\mathcal{U}$
is a universal cycle for $\bS$  we show that the three conditions of Theorem~\ref{secondTheorem} hold.
\begin{enumerate}
\item Clearly $|\alpha_1'|\geq n$.
\item $\alpha_1'$ has prefix $x^n$, which must be maximal since all of the universal cycles are disjoint.
\item We must show that consecutive strings in $\text{ext}_n(\alpha_1'),\text{ext}_n(\alpha_2'),\ldots, \text{ext}_n(\alpha_{m}')$ are suffix-related with respect to $(x,n)$. Notice that $\text{pre}_n(\text{ext}_n(\alpha_i')) = \text{pre}_n(\alpha_i)$ for $1\leq i \leq m$. So we only need to show that $(\text{pre}_n(\alpha_i),\text{pre}_n(\alpha_{i+1}))$ are suffix-related with respect to $(x,n)$ for $1\leq i < m$. The proof is by contradiction. Let $\text{pre}_n(\alpha_i) = a_1a_2\cdots a_n$ and $\text{pre}_n(\alpha_{i+1}) = b_1b_2\cdots b_n$ for some $1\leq i < m$. Let $j$ be the smallest index of $\alpha_{i+1}$ such that $b_j\neq x$. Suppose  $a_{j+1}a_{j+2} \cdots a_n \neq b_{j+1}b_{j+2}\cdots b_n$. Then there exists some largest $i> j$ such that $a_i \ne b_i$.  Since $\text{pre}_n(\alpha_i)$ comes before $\text{pre}_n(\alpha_{i+1})$ in colex order by Lemma~\ref{lemma:one}, $a_i<b_i$. However, since $\text{pre}_n(\alpha_{i+1})$ is a co-necklace, then $\gamma = 0^{i-1} b_i b_{i+1}\cdots b_n$ will also be a co-necklace. But this means $\gamma$ comes between $\text{pre}_n(\alpha_i)$ and $\text{pre}_n(\alpha_{i+1})$ in colex order, which is a contradiction.  So $a_{j+1}a_{j+2} \cdots a_n = b_{j+1}b_{j+2}\cdots b_n$,  which implies $\text{suff}_{n-j}(\text{pre}_n(\alpha_i))=\text{suff}_{n-j}(\text{pre}_n(\alpha_{i+1}))$. Thus the pair of strings $(\text{pre}_n(\alpha_i),\text{pre}_n(\alpha_{i+1}))$ are suffix-related with respect to $(x,n)$, which implies $(\alpha_i,\alpha_{i+1})$ are suffix-related with respect to $(x,n)$.
\end{enumerate}

\vspace{-0.2in}
%
%Thus, by Theorem~\ref{secondTheorem}, $\mathcal{U}$ is a universal cycle for $\bS$.
\end{proof}

When $m=|\mathbf{C}(n)|$, the above theorem yields the following corollary which describes a construction equivalent to the de Bruijn sequence construction from~\cite{Gabric2017}.  In that paper an algorithm
is provided that generates the de Bruijn sequence in $O(1)$-amortized time per bit.

\begin{corollary}\label{corollary:coNeckColex}
For $n \geq 2$, $\text{UC}(\text{revcolex}(\mathbf{coN}(n)))$ is a de Bruijn sequence for $\Sigma_2^n$.
\end{corollary}

\begin{exam} \small Consider the set $\mathbf{coN}(6)$ listed in reverse colex order:

\vspace{-0.25in}

\begin{align}
&000000111111, 000100111011, 001100110011, 000010111101, 001010110101, 000110111001.\nonumber
\end{align}

\vspace{-0.10in}

\noindent
By Corollary~\ref{corollary:coNeckColex}, the following is a de Bruijn sequence:

\vspace{-0.15in}

\[000000111111\cdot 000100111011\cdot 0011 \cdot 000010111101 \cdot 001010110101\cdot 000110111001.\]

\vspace{-0.05in}

\noindent
By Theorem~\ref{thm:coNeckColex}, $ 000000111111\cdot 000100111011\cdot 0011$,
is a universal cycle for $\factor_6(000000111111) \cup \factor_6(000100111011)\cup \factor_6(0011)$.
\end{exam}
%===================================================
\subsubsection{Rotations of Extended Co-necklaces in Lex Order}

Let $\mathbf{C}(n)$ be the set of all $\alpha = a_1a_2\cdots a_{2n}$ such that $a_{i+1}a_{i+2}\cdots a_{2n} a_1\cdots a_{i}\in \mathbf{coN}(n)$, where $i$ is the largest index of $\alpha$ such that $a_i\neq 0$. For example, $\mathbf{C}(5) =\{1111100000,1011101000,1001101100,1010101010\}$.

\begin{lemma}\label{lemma:two}
Let $\alpha=a_1a_2\cdots a_{2n}$ and $\beta=b_1b_2\cdots b_{2n}$ be in  $\mathbf{C}(n)$. If $(\text{pre}_n(\alpha),\text{pre}_n(\beta))$ are prefix-related with respect to $(x,n)$ where $x \in \{0,1\}$, then $(\alpha,\beta)$ are prefix-related with respect to $(1-x,n)$.
\end{lemma}
\begin{proof}
Assume that the pair of strings $(\text{pre}_n(\alpha),\text{pre}_n(\beta))$ are prefix-related with respect to $(x,n)$. There exists a smallest index $1 \leq i \leq n$ such that $x\neq a_{n-i}$ and $\text{pre}_{n-i-1}(\text{pre}_n(\alpha))=\text{pre}_{n-i-1}(\text{pre}_n(\beta))$. Since $\alpha = \text{pre}_n(\alpha)\overline{\text{pre}_n(\alpha)}$ and $\beta =\text{pre}_n(\beta)\overline{\text{pre}_n(\beta)}$, we have that the smallest index $j$ such that $1-x \neq a_{2n-j}$ is $j=i$ and we already know from before than $\text{pre}_{n-i-1}(\alpha)=\text{pre}_{n-i-1}(\beta)$. So $(\alpha,\beta)$ are prefix-related with respect to $(1-x,n)$.
\end{proof}

%%%THEOREM5

\begin{boxed2}
\vspace{-0.15in}
\begin{theorem}\label{thm:coNecklex}
For $n \geq 2$ and $m\geq 1$, let $\alpha_1,\alpha_2,\ldots, \alpha_m$ be the last $m$ strings in $\text{lex}(\mathbf{C}(n))$.  Then $\mathcal{U} = \text{UC}(\alpha_1,\alpha_2,\ldots, \alpha_m)$ is a universal cycle for $\bS  = \bigcup_{i=1}^m \factor_n(\alpha_i)$.
\end{theorem}

\vspace{-0.15in}
\end{boxed2}
\noindent
\begin{proof}
Observe that $\alpha_{m} = 1^n0^n$, so the last symbol in the sequence $\mathcal{U}$ is $x=0$. Let $\bS_i = \factor_n(\alpha_{i})$ and $\alpha_i'=pr(\alpha_i)$ for $1 \leq i\leq m$. Then $\bS_1,\bS_2,\ldots, \bS_{m}$ is a UC-partition of $\bS$ with universal cycles $\alpha_1',\alpha_2',\ldots, \alpha_m'$ where $x=0$. To prove that $\mathcal{U}$
is a universal cycle for $\bS$  we show that the three conditions of Corollary~\ref{corollaryOne} hold.
\begin{enumerate}
\item Clearly $|\alpha_m'|\geq n$.
\item $\alpha_m'$ has suffix $x^n$, which must be maximal since all of the universal cycles are disjoint.
\item We must show that consecutive strings in $\text{ext}_n(\alpha_1'),\text{ext}_n(\alpha_2'),\ldots, \text{ext}_n(\alpha_{m}')$ are prefix-related with respect to $(x,n)$. Notice that $\text{pre}_n(\text{ext}_n(\alpha_i')) = \text{pre}_n(\alpha_i)$ for $1\leq i \leq m$. By Lemma~\ref{lemma:two} we only need show that $(\text{pre}_n(\alpha_i),\text{pre}_n(\alpha_{i+1}))$ are prefix-related with respect to $(1-x,n)$ for $1\leq i < m$. The proof is by contradiction. Let $\text{pre}_n(\alpha_i) = a_1a_2\cdots a_n$ and $\text{pre}_n(\alpha_{i+1}) = b_1b_2\cdots b_n$ for some $1\leq i < m$. Let $j$ be the smallest index of $\text{pre}_n(\alpha_i)$ such that $a_{n-j}\neq 1-x$. Suppose $a_1a_2 \cdots a_{n-j-1} \neq b_1b_2 \cdots b_{n-j-1}$. Then there exists some smallest $s<n-j$ such that $a_{s}\neq b_{s}$. Since $\text{pre}_n(\alpha_i)$ comes before $\text{pre}_n(\alpha_{i+1})$ in lex order, then $a_{s}<b_{s}$. Let $\gamma=a_1a_2\cdots a_{s}1^{n-s}$. Clearly $\gamma$ is between $\text{pre}_n(\alpha_i)$ and $\text{pre}_n(\alpha_{i+1})$ in lex order, and $0^{n-s}a_1a_2\cdots a_{s}$ is a co-necklace since $0^{n-s}$ is the largest run of ones within the string, a contradiction. So $a_1a_2\cdots a_{n-j-1} = b_1b_2\cdots b_{n-j-1}$, which implies $\text{pre}_{n - j -1}(\text{pre}_n(\alpha_i)) = \text{pre}_{n - j -1}(\text{pre}_n(\alpha_{i+1}))$. Thus the pair of strings $(\text{pre}_n(\alpha_i),\text{pre}_n(\alpha_{i+1}))$ are prefix-related with respect to $(1-x,n)$, which implies $(\alpha_i,\alpha_{i+1})$ are prefix-related with respect to $(x,n)$.
\end{enumerate}

\vspace{-0.2in}
\end{proof}

When $m=|\mathbf{C}(n)|$, the above theorem yields the following corollary which describes a new de Bruijn sequence construction.

\begin{corollary}\label{corollary:coNecklex}
For $n\geq 2$, $\text{UC}(\text{lex}(\mathbf{C}(n)))$ is a de Bruijn sequence for $\Sigma_2^n$.
\end{corollary}

\begin{exam} \small Consider the set $\mathbf{C}(6)$ listed in lex order:

\vspace{-0.25in}

\begin{align}
&100111011000, 101011010100, 101111010000, 110011001100, 110111001000, 111111000000.\nonumber
\end{align}

\vspace{-0.10in}

\noindent
By Corollary~\ref{corollary:coNecklex}, the following is a de Bruijn sequence:

\vspace{-0.15in}

\[100111011000\cdot 101011010100\cdot 101111010000\cdot 1100\cdot 110111001000\cdot 111111000000.\]

\vspace{-0.05in}

\noindent
By Theorem~\ref{thm:coNecklex}, $1100 \cdot 110111001000\cdot 111111000000$,
is a universal cycle for $\factor_6(1100) \cup \factor_6(110111001000)\cup \factor_6(111111000000)$.
\end{exam}

\begin{comment}
\begin{table} [h]
  \begin{center} \small
   \caption{ Two de Bruijn sequences for $n=6$ using the co-necklace concatenation constructions (listed in the order they were proven).}
   
\label{table:example2}  
 \begin{tabular} {c    }
    {{\bf de Bruijn sequences for $n=6$}} \\  \hline

   			 $000000111111\cdot 000100111011\cdot 0011\cdot 000010111101\cdot 001010110101\cdot 000110111001$ \\
	 $100111011000\cdot 101011010100\cdot 101111010000\cdot 1100\cdot 110111001000\cdot 111111000000$

%1001000101000011101001101010110110010111001111110111100011000000
\end{tabular}

\end{center}
\end{table}
\end{comment}

\section{Conclusions and open problems}
In this paper we presented conditions for when a listing of universal cycles can be concatenated together to produce longer universal cycles and ultimately, de Bruijn sequences.  By applying the conditions, we generalized two previously known de Bruijn sequence concatenation-based constructions and discovered three new ones.  De Bruijn sequences from each of these five constructions for $n=6$ and $k=2$ are given below.

\small
\begin{center}
\begin{tabular}{l | c}
{\bf Construction} &  {\bf de Bruijn Sequence for $n=6,k=2$} \\   \hline
$\text{UC}(\text{lex}(\mathbf{Neck}_2(6)))$  &  0000001000011000101000111001001011001101001111010101110110111111  \\
$\text{UC}(\text{colex}(\mathbf{Neck}_2(6)))$  &   0000001001000101010011010000110010110110001110101110011110111111 \\

$\text{UC}(\text{revlex}(\mathbf{R}_2(6)))$  &  1111110111100111000110110100110000101110101100101010001001000000 \\
$\text{UC}(\text{revcolex}(\mathbf{coN}(6)))$  &  0000001111110001001110110011000010111101001010110101000110111001 \\
$\text{UC}(\text{lex}(\mathbf{C}(6)))$  &  1001110110001010110101001011110100001100110111001000111111000000  \\
\end{tabular}
\end{center}
\normalsize

\noindent
 We conclude by posing some open problems.
\begin{enumerate}
\item Can $k$-ary necklaces be generated in $O(n)$-amortized time per string  in colex order?   If so, then the de Bruijn sequence  $\text{UC}(\text{colex}(\mathbf{Neck}_k(n)))$ can be generated in $O(1)$-amortized time per symbol.  There is an efficient algorithm in the binary case~\cite{SAWADA201725}.
\item Can the de Bruijn sequences  $\text{UC}(\text{revlex}(\mathbf{R}_k(n)))$  and $\text{UC}(\text{lex}(\mathbf{C}(n)))$ be generated in $O(1)$-amortized time per symbol?
%\item Can we generalize this result further to include the cool-lex construction?
\end{enumerate}

%----------- BIBLIOGRAPHY --------------

\small
\bibliographystyle{abbrv}
\bibliography{refs}

\begin{thebibliography}{10}

\bibitem{pref-opposite}
A.~Alhakim.
\newblock A simple combinatorial algorithm for de {B}ruijn sequences.
\newblock {\em The American Mathematical Monthly}, 117(8):728--732, 2010.

\bibitem{alhakim-span}
A.~Alhakim.
\newblock Spans of preference functions for de {B}ruijn sequences.
\newblock {\em Discrete Applied Mathematics}, 160(7-8):992 -- 998, 2012.

\bibitem{grandma2}
P.~B. Dragon, O.~I. Hernandez, J.~Sawada, A.~Williams, and D.~Wong.
\newblock Constructing de {B}ruijn sequences with co-lexicographic order: The
  $k$-ary {G}randmama sequence.
\newblock {\em European Journal of Combinatorics, to appear}, 2017.

\bibitem{grandma}
P.~B. Dragon, O.~I. Hernandez, and A.~Williams.
\newblock The grandmama de {B}ruijn sequence for binary strings.
\newblock In {\em Proceedings of LATIN 2016: Theoretical Informatics: 12th
  Latin American Symposium, Ensenada, Mexico}, pages 347--361. Springer Berlin
  Heidelberg, 2016.

\bibitem{eldert}
C.~Eldert, H.~Gray, H.~Gurk, and M.~Rubinoff.
\newblock Shifting counters.
\newblock {\em AIEE Trans.}, 77:70--74, 1958.

\bibitem{ETZION1986331}
T.~Etzion.
\newblock An algorithm for constructing m-ary de {B}ruijn sequences.
\newblock {\em Journal of Algorithms}, 7(3):331 -- 340, 1986.

\bibitem{Etzion1987}
T.~Etzion.
\newblock Self-dual sequences.
\newblock {\em Journal of Combinatorial Theory, Series A}, 44(2):288 -- 298,
  1987.

\bibitem{Etzion1984}
T.~Etzion and A.~Lempel.
\newblock Construction of de {B}ruijn sequences of minimal complexity.
\newblock {\em IEEE Transactions on Information Theory}, 30(5):705--709,
  September 1984.

\bibitem{ford}
L.~Ford.
\newblock A cyclic arrangement of ${M}$-tuples.
\newblock {\em Report No. P-1071, Rand Corporation, Santa Monica, California,
  April 23}, 1957.

\bibitem{fred-lexleast}
H.~Fredricksen.
\newblock The lexicographically least de {B}ruijn cycle.
\newblock {\em Journal of Combinatorial Theory}, 9:1--5, 1970.

\bibitem{fred-nfsr}
H.~Fredricksen.
\newblock A survey of full length nonlinear shift register cycle algorithms.
\newblock {\em Siam Review}, 24(2):195--221, 1982.

\bibitem{fkm2}
H.~Fredricksen and J.~Maiorana.
\newblock Necklaces of beads in $k$ colors and $k$-ary de {B}ruijn sequences.
\newblock {\em Discrete Math.}, 23:207--210, 1978.

\bibitem{Gabric2017}
D.~Gabric and J.~Sawada.
\newblock A de {B}ruijn sequence construction by concatenating cycles of the
  complemented cycling register.
\newblock In {\em Combinatorics on Words - 11th International Conference,
  {WORDS} 2017, Montr{\'{e}}al, QC, Canada, September 11-15, 2017,
  Proceedings}, pages 49--58, 2017.

\bibitem{golomb}
S.~W. Golomb.
\newblock {\em Shift Register Sequences}.
\newblock Aegean Park Press, Laguna Hills, CA, USA, 1981.

\bibitem{huang}
Y.~Huang.
\newblock A new algorithm for the generation of binary de {B}ruijn sequences.
\newblock {\em J. Algorithms}, 11(1):44--51, 1990.

\bibitem{jansen}
C.~J.~A. Jansen, W.~G. Franx, and D.~E. Boekee.
\newblock An efficient algorithm for the generation of de {B}ruijn cycles.
\newblock {\em IEEE Transactions on Information Theory}, 37(5):1475--1478, Sep
  1991.

\bibitem{martin}
M.~H. Martin.
\newblock A problem in arrangements.
\newblock {\em Bull. Amer. Math. Soc.}, 40(12):859--864, 1934.

\bibitem{ruskey}
F.~Ruskey.
\newblock {\em Combinatorial Generation}.
\newblock Working version (1i) edition, 1996.

\bibitem{RSW}
F.~Ruskey, C.~Savage, and T.~M.~Y. Wang.
\newblock Generating necklaces.
\newblock {\em J. Algorithms}, 13:414--430, 1992.

\bibitem{coollex}
F.~Ruskey, J.~Sawada, and A.~Williams.
\newblock De {B}ruijn sequences for fixed-weight binary strings.
\newblock {\em SIAM Journal on Discrete Mathematics}, 26(2):605--617, 2012.

\bibitem{walcom}
J.~Sawada, B.~Stevens, and A.~Williams.
\newblock De {B}ruijn sequences for the binary strings with maximum specified
  density.
\newblock In {\em Proceeding of 5th International Workshop on Algorithms and
  Computation (WALCOM 2011), New Dehli, India, LNCS}, 2011.

\bibitem{dbrange}
J.~Sawada, A.~Williams, and D.~Wong.
\newblock Universal cycles for weight-range binary strings.
\newblock In {\em Combinatorial Algorithms - 24th International Workshop,
  {IWOCA} 2013, Rouen, France, July 10-12, 2013, Revised Selected Papers},
  pages 388--401, 2013.

\bibitem{SAWADA201431}
J.~Sawada, A.~Williams, and D.~Wong.
\newblock The lexicographically smallest universal cycle for binary strings
  with minimum specified weight.
\newblock {\em Journal of Discrete Algorithms}, 28:31 -- 40, 2014.
\newblock StringMasters 2012 and 2013 Special Issue (Volume 1).

\bibitem{wong}
J.~Sawada, A.~Williams, and D.~Wong.
\newblock A surprisingly simple de {B}ruijn sequence construction.
\newblock {\em Discrete Math.}, 339:127--131, 2016.

\bibitem{SAWADA201725}
J.~Sawada, A.~Williams, and D.~Wong.
\newblock Necklaces and {L}yndon words in colexicographic and binary reflected
  {G}ray code order.
\newblock {\em Journal of Discrete Algorithms}, 46-47(Supplement C):25 -- 35,
  2017.

\bibitem{kwong}
J.~Sawada, A.~Williams, and D.~Wong.
\newblock A simple shift rule for k-ary de bruijn sequences.
\newblock {\em Discrete Mathematics}, 340(3):524 -- 531, 2017.

\end{thebibliography}

\end{document}